\RequirePackage{fix-cm}
\documentclass[smallextended]{svjour3}       
\smartqed  
\usepackage{graphicx}

\usepackage{amsmath,amsfonts,graphicx}
\usepackage{color}

\newcommand{\I}{\mathcal{I}}
\newcommand{\J}{\mathcal{J}}
\newcommand{\Hdiv}[1]{H(\mathrm{div},{#1})}
\def\d{\partial}
\newcommand{\RRR}{\mathbb{R}}

\newcommand{\eq}{\boldsymbol{\varepsilon}_{\bf q}}
\newcommand{\eqhat}{\boldsymbol{\varepsilon}_{\widehat{q}}}
\newcommand{\eu}{\varepsilon_u}

\newcommand{\el}{\varepsilon_{\widehat u}}
\newcommand{\vPhi}{\boldsymbol{\varPhi}}
\newcommand{\vTheta}{\varTheta}
\newcommand{\vPsi}{\varPsi}

\newcommand{\Eh}{\mathcal{E}_h}
\newcommand{\oh}{{\mathcal{T}_h}}
\newcommand{\doh}{{\partial \oh}}
\newcommand{\Th}{{\oh}}
\newcommand{\qhat}{\widehat{\boldsymbol{q}}}

\newcommand{\uhat}{\widehat{u}}

\newcommand{\ip}[1]{\langle {#1} \rangle}

\newcommand{\pv}{\boldsymbol{\varPi}_{\!\scriptscriptstyle{V}}}
\newcommand{\pw}{\varPi_{\!\scriptscriptstyle{W}}}
\newcommand{\Pw}{P_{\scriptscriptstyle{W}}}
\newcommand{\Pm}{P_{\scriptscriptstyle{M}}}

\newcommand{\vq}{{\boldsymbol{q}}}

\newcommand{\vr}{{\boldsymbol{r}}}

\newcommand{\vn}{{\boldsymbol{n}}}
\newcommand{\vv}{{\boldsymbol{v}}}
\newcommand{\vw}{{\boldsymbol{w}}}
\newcommand{\dive}{\mathop{\nabla}\cdot\,}
\newcommand{\pol}{\mathcal{P}}

\title
{Convergence and superconvergence analyses of HDG  methods  for time fractional diffusion problems \thanks{Support of the King Fahd University of Petroleum and Minerals
(KFUPM) through
 the project FT131011 is gratefully acknowledged.}}
\author{
 Kassem  Mustapha \and {Maher Nour} \and Bernardo Cockburn }
\institute{K. Mustapha ~\email{kassem@kfupm.edu.sa}
           \and
 {M. Nour}~ \email{mnoor@kfupm.edu.sa} \at
 Department of Mathematics and Statistics, KFUPM, Saudi Arabia
  \and B. Cockburn ~\email{cockburn@math.umn.edu}
  \at School of Mathematics, University of Minnesota, USA }
\date{\today}
\begin{document}
\maketitle
\begin{abstract}
We study  the hybridizable discontinuous Galerkin (HDG) method for the spatial discretization of time fractional diffusion models with Caputo derivative of order $0<\alpha<1$. For each time $t \in [0,T]$, the HDG approximations are taken to be piecewise polynomials of
degree $k\ge0$ on the  spatial domain~$\Omega$, the approximations to the exact solution  $u$ in the $L_\infty\bigr(0,T;L_2(\Omega)\bigr)$-norm   and to $\nabla u$  in the $L_\infty\bigr(0,T;{\bf L}_2(\Omega)\bigr)$-norm are
 proven to converge with the rate $h^{k+1}$ provided that $u$ is sufficiently regular, where $h$ is the maximum diameter of the
elements of the mesh.  Moreover, for $k\ge1$, we obtain a superconvergence result which allows us to
compute, in an elementwise manner, a new approximation for $u$
converging with a rate  $h^{k+2}$ (ignoring the logarithmic factor), for quasi-uniform spatial meshes. Numerical experiments
validating the theoretical results are displayed.
\end{abstract}




\section{Introduction}
In this paper, we study the method resulting after using exact integration in time and a  hybridizable discontinuous Galerkin (HDG) method for the spatial discretization of the following time fractional diffusion model problem:
\begin{subequations}
\label{eq: time-fractional diffusion models}
\begin{alignat}{2}\label{time-fractional diffusion models-a}
&^{c}{\rm D}^{1- \alpha}u(x,t)- \Delta u(x,t) = f(x,t)
&&\quad\mbox{ for } (x,t)\in \Omega\times (0,T],
\\ \label{time-fractional diffusion models-b}
&u(x,t)= g(x) &&\quad\mbox{ for } (x,t)\in \partial\Omega\times (0,T],
\end{alignat}
\end{subequations}
with $u(x,0)=u_0(x)$ for $x\in \Omega,$
where $\Omega$ is a convex polyhedral domain of $\mathbb{R}^d$ ($d=1,2,3$) with boundary $\partial \Omega$, $f$, $g$ and $u_0$ are given
functions assumed to be sufficiently regular such that the solution $u$ of \eqref{eq: time-fractional diffusion models} is in the space $W^{1,1}(0,T;H^2(\Omega))$, see  the regularity analysis in \cite{McLean2010} (further regularity assumptions will be imposed later),  and $T>0$ is a fixed
but arbitrary value.
Here, $^{c}{\rm D}^{1- \alpha}$  denotes time fractional Caputo derivative of order $\alpha$ defined by
\begin{equation}  ^{c}{\rm D}^{1-\alpha} v(t):= \I^{\alpha} v' (t):=\int_0^t\omega_{\alpha}(t-s)v'(s)\,ds\quad {\rm with}~~
0<\alpha<1,
\end{equation}
where $v'$ denotes  the time derivative of the function $v$ and $\I^\alpha$ is the Riemann--Liouville (time) fractional integral operator; with  $\omega_{\alpha}(t):=\frac{t^{\alpha-1}}{\Gamma(\alpha)}$ and $\Gamma$ being  the  gamma function.
 
 In this work, we investigate  a high-order accurate numerical method for the space
discretization for problem \eqref{eq: time-fractional diffusion models}. Using exact integration in time, we propose to deal with the accuracy issue by developing   a high-order HDG method that allows for locally varying spatial meshes and   approximation orders which are beneficial to handle problems with low regularity.    The HDG methods were introduced in
\cite{CockburnGopalakrishnanLazarov09} in the framework of steady-sate diffusion which share with the classical (hybridized version of the)
mixed  finite element methods their remarkable convergence and superconvergence
properties, \cite{CockburnQiuShi}, as well as the
way in which they can be efficiently implemented,
\cite{KirbySherwinCockburn}. They provide approximations that are more accurate
than the ones given by any other DG method for second-order elliptic problems
\cite{NguyenPeraireCockburnIcosahom09}.
 In \cite{CockburnMustapha2014}, a similar method was studied for  the fractional subdiffusion problem:
\begin{equation}
\label{eq: reimann} u'(x,t) - {\rm D}^{1-\alpha} \Delta u(x,t)  = f(x,t) \quad\mbox{ for } (x,t)\in \Omega\times (0,T],
\end{equation}
where  ${\rm D}^{1-\alpha}$ is the Riemann--Liouville fractional time derivative operator,  \begin{equation}\label{eq: fractional derivative}
{\rm D}^{1-\alpha} v(t):=( \I^\alpha v(t))'=\frac{\partial}{\partial t}\int_0^t \omega_\alpha(t-s)\,v(s)\,ds\,.
 \end{equation}
 (For other numerical methods of \eqref{eq: reimann}, see  \cite{ChenLiuAnhTurner2011,CuestaLubichPalencia2006,Cui2009,McLeanMustapha2014,Mustapha2014,MustaphaAlMutawa,MustaphaMcLean2012,MustaphaMcLean2013,YusteQuintana2009} and related references therein.) When $f \equiv 0$ (that is, homogeneous case), 
the two representations \eqref{time-fractional diffusion models-a} and \eqref{eq: reimann} are different ways of writing the same equation, as they are equivalent under reasonable assumptions on the
initial data. However, the numerical methods obtained for each representation are formally different.
      In \cite{CockburnMustapha2014}, the authors extended the approach of the error analysis used in \cite{ChabaudCockburn12} for the heat equation by using several important properties of  ${\rm D}^{1-\alpha}$.  Indeed, a duality argument was applied (where delicate regularity estimates were required) to prove
 the superconvergence properties of the method.

We start our work by introducing the spatial semi-discrete HDG method for the model problem \eqref{eq: time-fractional diffusion models} in the next section. For sake of implementing the HDG scheme, we discretize in time using a generalized Crank-Nicolson scheme \cite{MustaphaAbdallahFurati2014}. The existence and uniqueness of the obtained fully discrete scheme will be shown.   In Section~\ref{sec: errors}, we prove the main optimal convergence results of the HDG method. Indeed, for each time $t \in [0,T]$, we prove that the error of the HDG approximation to the solution $u$ of \eqref{eq: time-fractional diffusion models} in the $L_\infty\bigr(0,T;{L}_2(\Omega)\bigr)$-norm and to the flux ${\bf q}:=-\nabla u$  in the $L_\infty\bigr(0,T;{\bf L}_2(\Omega)\bigr)$-norm converge with order  $h^{k+1}$ where $k$ is the polynomial degree and $h$ is the maximum diameter of the
elements of the spatial mesh; see   Theorem \ref{thm: optimal errors}.  Some important properties of the fractional integral operator $\I^\alpha$ are used in our a priori error analysis. In Section~\ref{sec:super},  for quasi-uniform meshes and whenever  $k\ge1$, by a simple elementwise postprocessing with a computation cost that is negligible in comparison with that of obtaining the HDG approximate solution, we obtain a better approximation to $u$ converging in the $L_\infty \bigr(0,T;L_2(\Omega)\bigr)$-norm with a rate of order  $\sqrt{\log (T/h^{2/(\alpha+1)}) }h^{k+2}$; see Theorem \ref{thm:super}. Here,  we partially rely on the superconvergence analysis of the postprocessed HDG scheme in \cite[Section 5]{CockburnMustapha2014}. In Section~\ref{sec: numerics}, we  present some numerical tests which indicate the validity of our theoretical optimal convergence rates of the HDG scheme as well as the superconvergence rates of  the postprocessed HDG scheme.

Here is a brief history of the  numerical methods for  problem \eqref{eq: time-fractional diffusion models} in the existing literature.  For the {\em one
dimensional case},  a box-type scheme based on combining order reduction approach and an L$_1$-discretization was
considered in \cite{ZhaoSun2011}. An explicit finite difference (FD) method, we refer the reader to
\cite{Quintana-MurilloYuste2011}. For an implicit FD scheme in time and Legendre spectral methods in space were
studied in \cite{LinXu2007}.  An extension of this work was considered in \cite{LiXu2009}, where  a time-space spectral method has been
proposed and analyzed.  An implicit
Crank--Nicolson had been considered in \cite{SweilamKhaderMahdy2012} where the stability of the proposed scheme was proven.
Two finite difference/element approaches were developed in \cite{ZengoLiLiuTurner2013}. Therein, the time direction was approximated by the fractional linear multistep method and the space direction was approximated by the standard finite element method (FEM).  A compact difference scheme (fourth order in space) was proposed in \cite{ZhaoXu2014}  for solving problem \eqref{eq: time-fractional diffusion models} but with a variable diffusion parameter.  The unconditional stability and the global
convergence of the scheme were shown.  In \cite{XuZheng2013}, a high-order local DG (LDG) method for space discretization was studied. Optimal convergence rates was proved.

For the {\em two- (or three-)  dimensional cases},  a standard second-order central difference approximation was used in space, and, for the time stepping, two
alternating direction implicit schemes ($L_1$-approximation and backward Euler method)  were investigated in
\cite{ZhangSun2011}.  A fractional alternating direction implicit scheme for problem \eqref{eq: time-fractional diffusion models} in 3D was proposed in \cite{ChenLiuLiuChenAnhTurnerBurrage2014}. Unique solvablity, unconditional stablity and convergence in $H^1$-norm were shown.
 A compact fourth order FD method (in space) with operator-splitting techniques was considered in
\cite{Cui2012}.  The Caputo derivative was evaluated by the $L_1$ approximation, and the second order spatial derivatives
  were approximated by the fourth-order, compact (implicit) finite differences. In \cite{JinLazarovZhou2014}, the authors  developed two simple fully discrete schemes based on piecewise linear Galerkin FEMs in space and implicit backward differences for the  time discretizations. Finally,   a high-order accurate (variable) time-stepping discontinuous Petrov-Galerkin  that allows low regularity combined with standard finite elements in space was investigated recently  in \cite{MustaphaAbdallahFurati2014}. Stability and error analysis were rigourously studied.

\section{The HDG method}\label{sec: HDG scheme}
 This section is devoted  to defining a scalar approximation $u_h(t)$ to $u(t)$,
 a vector approximation  $\vq_h(t)$  to the flux  $\vq(t)$,
and {a scalar approximation $\uhat_h(t)$ to the {trace} of $u(t)$ {on} element {boundaries}} for each time $t \in [0,T],$ using
 a spatial HDG method. We begin by discretizing the domain $\Omega$ by a conforming
 triangulation (for simplicity) $\oh$ made of simplexes $K$; we denote by $\partial
 \oh$ the set of all the boundaries $\partial K$ of the elements $K$ of $\oh$.  We denote by $\Eh$ the union of faces $F$ of the simplexes $K$ of the triangulation
$\oh$.

Next, we introduce the discontinuous finite element spaces:
\begin{subequations}
\begin{alignat}{3}
  \label{eq:W}
  {W_h} &= \{ {w \in L^2(\Omega)}&&{: w|_K \in \pol_k(K)}&&\;\;\;{\forall\;K\in\oh}\},
  \\
  \label{eq:V}
  {\boldsymbol{V}_h} &=  \{ \vv\in [L_2(\Omega)]^d&&{:\; \vv|_K \in [\pol_k(K)]^d}&&\;\;\;{\forall\; K\in\oh}\},
  \\
  \label{eq:M}
  M_h &= \{ \mu\in L^2(\Eh)&&: \;\mu|_F \in
  \pol_k(F)&&\;\;\;\forall\; F\in \Eh\},
\end{alignat}
\end{subequations}
 where $\pol_k(K)$ is the space of
polynomials of total degree at most $k$ in the spatial variable.

To describe our scheme,  we rewrite (\ref{time-fractional
diffusion models-a}) as a first order system as follows:
$\vq
+\nabla u = 0$ and $^{c}{\rm D}^{1- \alpha}u+\dive
\vq = f.$ So,
$\vq$ and $u$ satisfy: for  $t \in (0,T]$,
\begin{subequations}
\label{eq: weak exact solution}
  \begin{alignat}{2}\label{eq: weak exact solution1}
    ( \vq,\boldsymbol{\phi}) -
    (  u,\dive \boldsymbol{\phi}) +   \ip{ u,\boldsymbol{\phi}\cdot\vn}
    &  = 0 && \quad \forall \,\boldsymbol{\phi}\in { \Hdiv\Omega},
    \\\label{eq: weak exact solution3}
     (^{c}{\rm D}^{1- \alpha}u,\chi) -(   \vq, \nabla \chi) +
    \ip{  \vq\cdot\vn, \chi}
    & =
    (f, \chi)&&\quad \forall \, \chi \in H^1(\Omega)\,.\end{alignat}
\end{subequations}
where  $(v,w):=\sum_{K \in \Th}(v,w)_K$ and $\ip{v,w} := \sum_{K \in \Th} \ip{ v,w}_{\d K}$. Throughout the paper, for any domain
$D$ in $\RRR^d$, by  $(u,v)_D =\int_D u v \; dx$ we denote the $L_2$-inner product on $D$. However,  we use instead $\ip{u,v}_{D}$ for the
$L_2$-inner product when $D$ is a domain of $\RRR^{d-1}$. We use  $\| \cdot\|_D$ to denote the $L^2(D)$-norm where we drop $D$ when
$D=\Omega$. For vector functions $\vv$ and $\vw$, the notation is similarly defined
with the integrand being the dot product $\vv \cdot \vw$.
 For later use, the norm and semi-norm  on any Sobolev space $X$ are denoted by $\|\cdot\|_X$ and $|\cdot|_X$, respectively. We also  denote $\| \cdot\|_{X(0,T;Y(\Omega))}$ by $\| \cdot\|_{X(Y)}$.

For each $t>0$, the HDG method {provides} approximations $u_h(t) \in W_h$, $\vq_h(t) \in
\boldsymbol{V}_h$, and $\uhat_h(t) \in M_h$ of $u(t)$, $\vq(t)$, and the trace of $u(t)$, respectively. These are
determined by requiring that
\begin{subequations}
\label{method}
\begin{alignat}{2}
\label{eq:method-a}
 ( \vq_h,\vr) -
 ( u_h,\dive \vr) +   \ip{\uhat_h,\vr\cdot\vn}
 &  = 0, \quad &&  \forall \,\vr\, {\in} \boldsymbol{V}_h,
 \\
 \label{eq:method-b}
(^{c}{\rm D}^{1- \alpha}u_h, w) -( \vq_h, \nabla w) +
 \ip{ \widehat{\boldsymbol{q}}_h\cdot\vn, w}
 & =
 (f, w),  \quad &&\forall \, w \in W_h,
\\
\label{eq:method-d}
  \ip{\uhat_h,\mu}_{\d\Omega}&=\ip{g,\mu }_{\d\Omega}, \quad &&\forall \,\mu \in M_h,
\\
\label{eq:method-c}
 \ip{\widehat{\boldsymbol{q}}_h \cdot \vn,\mu }-\ip{\widehat{\boldsymbol{q}}_h \cdot \vn,\mu }_{ \d\Omega} & =
 0, \quad &&\forall \, \mu \in M_h,
\end{alignat}
     and take the
numerical trace for the flux as
\begin{alignat}{2}
\label{HDGtrace}  \qhat_h &= \vq_h + \tau \, \big( u_h - \uhat_h \big) \vn  &&\quad \text{ on } \doh,
\end{alignat}
\end{subequations}
for some nonnegative stabilization function $\tau$ defined on
$\doh$; we assume that, for each element $K\in\oh$, $\tau|_{\partial
K}$ is constant on each of its faces. At $t=0$, $u_h(0)  := \pw u_0$ where the operator $\pw$ is defined in \eqref{eq:proj}.

The first two
equations are inspired in the weak form of the fractional
differential equations satisfied by the exact solution,  \eqref{eq:
weak exact solution}.  The form of the numerical
trace given by \eqref{eq:method-c} allows us to express
$(u_h,\vq_h,\widehat{\boldsymbol{q}}_h)$ elementwise in terms of
$\uhat_h$ and $f$  by using equations \eqref{eq:method-a},
\eqref{eq:method-b} and \eqref{HDGtrace}. Then,
the numerical trace  $\uhat_h$ is determined by as the solution of the transmission
condition \eqref{eq:method-c}, which enforces the single-valuedness
of the normal component of the numerical trace
$\widehat{\boldsymbol{q}}_h$, and the boundary condition
\eqref{eq:method-d}. Thus, the only globally-coupled degrees of
freedom are those of $\uhat_h$.

 In our experiments, to implement our spatial semi-discrete HDG scheme \eqref{method},   we use for simplicity  a generalized Crank-Nicolson (CN) scheme for time discretization, see \cite{Mustapha2011,MustaphaAbdallahFurati2014}.
Formally, the CN scheme is second-order accurate provided that the continuous solution is sufficiently regular.
To this end,   we introduce a  uniform partition of the time interval
$[0,T]$ given by the points: $t_i= i\delta $ for $i=0,\cdots,N,$ with $\delta=T/N$ being the time-step size.  We take $\delta$ to be sufficiently small  so that the spatial discretizations errors  are dominant.

 The time-stepping CN combined with the above HDG method {provides} approximations $u_h^j \in W_h$, $\vq_h^j \in
\boldsymbol{V}_h^j$, and  $\uhat_h^j \in M_h$ of $u(t_j)$, $\vq(t_j)$, and the trace of $u(t_j)$, respectively, for $j=1,\cdots,N$. Starting from $u_h^0  = \pw u_0$, and with appropriate choices of $\vq_h^0$ and $\uhat_h^0$, our fully discrete scheme is defined by:
\begin{equation}
\label{Fully discrete scheme}
\begin{aligned}
 ( \vq_h^{j-\frac{1}{2}},\vr) -
 ( u_h^{j-\frac{1}{2}},\dive \vr) +   \ip{\uhat_h^{j-\frac{1}{2}},\vr\cdot\vn}
 &  = 0,&&~~ \forall \,\vr  \in \boldsymbol{V}_h,
 \\
(\J_\alpha\overline u_h(t_j), w)-( \vq_h^{j-\frac{1}{2}}, \nabla w) +
 \ip{ \widehat{\boldsymbol{q}}_h^{j-\frac{1}{2}}\cdot\vn, w}
 & =
 (f^{j-\frac{1}{2}}, w),&&~~ \forall\, w \in W_h,
\\
\ip{\uhat_h^j,\mu}_{\d\Omega}&=\ip{g,\mu }_{\d\Omega}, &&~~\forall \,\mu \in M_h,
\\
 \ip{\widehat{\boldsymbol{q}}_h^j \cdot \vn,\mu_1 }-\ip{\widehat{\boldsymbol{q}}^j_h \cdot \vn,\mu_1 }_{ \d\Omega} & =
 0,&&~~ \forall \,\mu_1 \in M_h,
\end{aligned}
\end{equation}
 where $f^{j-\frac{1}{2}}:= \frac{1}{2}(f(t_{j-1})+f(t_j))$,  $
  \qhat_h^j = \vq_h^j + \tau \, \big( u_h^j - \uhat_h^j \big) \vn$ on $\doh$,
  \begin{align*}\J_\alpha\overline u_h(t_j)&=\int_{t_{j-1}}^{t_j} \int_0^t\omega_\alpha(t-s)\overline u_h(s)\,ds\,dt,\end{align*}
with $\overline u_h(s):= \delta^{-1}(u^i_h-u^{i-1}_h)$ for $s\in (t_{i-1}, t_i)$,   $ \vq_h^{j-\frac{1}{2}}:= \frac{1}{2}( \vq_h^j+ \vq_h^{j-1})$, and the functions $u_h^{j-\frac{1}{2}}$, $\uhat_h^{j-\frac{1}{2}}$, and $\widehat{\boldsymbol{q}}_h^{j-\frac{1}{2}}$ are similarly defined.

For each $1\le j\le N$, \eqref{Fully discrete scheme} amounts to a square linear system. Thus the existence of the CN HDG solution follows
 from its uniqueness. We prove the uniqueness by induction hypothesis on $j$. We let $f^{i-\frac{1}{2}}$ (for $1\le i\le j$) and $g$ be identically zero  in \eqref{Fully discrete scheme},  we assume that  $(u_h^i,\vq_h^i,\uhat_h^i)\equiv (0,{\bf
0},0)$ for $1\le i\le j-1$ and  the task is to show that this holds true for $i=j$. To do so, choose $\vr =\vq_h^j$,   $w = u_h^{j}$,  $\mu = \widehat{\boldsymbol{q}}_h^j \cdot \vn$ and $\mu_1=\uhat_h^j$ in \eqref{Fully discrete scheme} and then simplify, yield
\begin{equation*}
\begin{aligned}
 \|\vq_h^j\|^2 -
 ( u_h^{j},\dive \vq_h^j) +   \ip{\uhat_h^{j},\vq_h^j\cdot\vn}
 &  = 0,
 \\
(\J_{\alpha}\overline u_h(t_j), u_h^{j}) -( \vq_h^{j}, \nabla u_h^{j}) +
 \ip{ \widehat{\boldsymbol{q}}_h^{j}\cdot\vn, u_h^{j}}
 & =0,
\\
 \ip{\widehat{\boldsymbol{q}}_h^j \cdot \vn,\uhat_h^j }&=0
 ,
\end{aligned}
\end{equation*}
Since $( u_h^{j},\dive \vq_h^j) = \ip{u_h^{j},\vq_h^j\cdot\vn} -( \vq_h^{j}, \nabla u_h^{j}),$ adding the above equations give
\[
(\J^{\alpha}\overline u_h(t_j), u_h^{j})+\|\vq_h^j\|^2
 + \ip{\uhat_h^{j}-u_h^{j},(\vq_h^j-\widehat{\boldsymbol{q}}_h^{j})\cdot\vn}=0\,.\]
Hence, by  the induction hypothesis and the identity $(\vq_h^j-\widehat{\boldsymbol{q}}_h^{j})\cdot\vn = \tau \, \big( u_h^j - \uhat_h^j \big) $ on $\doh$, we notes that
\[
\int_0^{t_{j}}(\I^{\alpha}\overline u_h(t), \overline u_h(t))\,dt+\|\vq_h^j\|^2
 + \|\sqrt{\tau}(\uhat_h^{j}-u_h^{j})\|_{\doh}^2=0,\]
and therefore,  the use of the coercivity property of $\I^\alpha$ (see \eqref{eq: coercivity}) completes the proof.

\section{Error estimates}\label{sec: errors}
In this section, we carry our a priori error
analysis of the HDG method. Following \cite{ChabaudCockburn12,CockburnGopalakrishnanSayas09,CockburnMustapha2014}, we start by defining  the projections below which play the comparison function role in the error analysis.

For each $t\in (0,T]$, we assume that $\vq(t)\in [H^1(\oh)]^d$ and $u(t)\in  H^1(\oh)$, where
 $H^1(\oh) = \prod_{K\in \oh} H^1(K),$  the projections $ \pv\vq(t) \in \boldsymbol{V}_h$ and $\pw u(t) \in W_h$
are defined by: on each simplex $K\in\oh$ and for all faces $F$ of $K$,
\begin{subequations}
  \label{eq:proj}
  \begin{align}
  \label{eq:proj1 new}
    (\pv \vq(t), \vv)_K
    &= ( \vq(t), \vv)_K,
    \\
    \label{eq:proj2}
    (\pw u(t), w )_K
    &= ( u(t), w )_K,
    \\
    \label{eq:proj3}
    \ip{\pv \vq(t) \cdot \vn + \tau \pw u(t),\mu }_F
    &= \ip{\vq(t)\cdot\vn + \tau  u(t),\mu }_F,
  \end{align}
\end{subequations}
for al $\vv \in [\pol_{k-1}(K)]^d$,  $w \in \pol_{k-1}(K)$ and $\mu \in \pol_k(F)$. This projection
introduced in \cite{CockburnGopalakrishnanSayas09} to study HDG methods for the steady-state diffusion
problem and also used in the error analyses of HDG methods for classical diffusion \cite{ChabaudCockburn12} as well as  for fractional subdiffusion \cite{CockburnMustapha2014} problems. Its approximation properties are described in the following
result.
\begin{theorem} (\cite{CockburnGopalakrishnanSayas09})
\label{thm:proj} Suppose
   {$\tau|_{\d K}$ is nonnegative and $\tau^{\max}_K:=\max\tau|_{\d K} >0$}.
  Then the system~\eqref{eq:proj} is uniquely solvable for
$\pv \vq$ and~$\pw u$.  Furthermore, there is a constant ${C}$
  independent of $K$ and $\tau$ such that for each $t\in (0,T]$,
\begin{subequations}
  \begin{alignat*}{1}
   \|e_\vq(t) \|_K\le&\;C\,h^{k+1}_K\left(|\vq(t)|_{\boldsymbol{H}^{k+1}(K)}
                       +{\tau_K^{*}}\,| u(t)|_{H^{k+1}(K)}\right),
    \\
    \|  e_u(t)\|_K \le&\;C\,h^{k+1}_K\left(|u(t)|_{H^{k+1}(K)}
                       +|\nabla\cdot\vq(t)|_{H^k(K)}/\tau_K^{\max}\right)
  \end{alignat*}
 \end{subequations}
where $e_\vq:= \pv \vq-\vq$ and $e_u:= \pw  u - u.$  Here $\tau_K^{*}:=\max \tau|_{\d K\setminus F^*}$, where $F^*$ is a face of $K$ at which $\tau|_{\d K}$ is maximum.
 \end{theorem}
Note that  the approximation error of the projection is of order $k+1$  provided that  the stabilization function is such that both  $\tau^*_K$
and $1/\tau^{\max}_K$ are uniformly bounded and the exact solution is sufficiently regular.

Thus, the main task now is to estimate the terms $\eu:=\pw u-u_h$ and  $\eq:= \pv \vq-\vq_h$.   For convenience, we further introduce the following notations:
$\el:= P_M u-\widehat{u}_h$ and  $\eqhat:=   \boldsymbol{P}_M \vq-\widehat{\vq}_h$
where $P_M$ denotes the $L^2$-orthogonal projection onto $M_h$, and $\boldsymbol{P}_M$ denotes the vector-valued projection each of  whose
components are equal to $P_M$.

The projection of the errors satisfy the equations stated in the next lemma.
\begin{lemma}        \label{lem:consistency}
For each $t >0,$ we have
  \begin{subequations}
    \label{eq:err}
    \begin{alignat}{2}
      \label{eq:err-a}
 ( \eq,\vr) - ( \eu,\dive \vr) +   \ip{ \el,\vr\cdot\vn} &  = ( e_{\vq}
 ,\vr),\quad &&\forall\, \vr\in \boldsymbol{V}_h
 \\
      \label{eq:err-b}
 (\I^\alpha\eu',w)-( \eq, \nabla w) +   \ip{ \eqhat \cdot\vn, w}
 & = (\I^\alpha e_u',w),\quad &&\forall\, w\in W_h
\\
    \label{eq:err-c}
 \ip{\el,\mu}_{\d\Omega}
   & = 0,\quad &&\forall\, \mu\in M_h
\\
    \label{eq:err-d}
 \ip{\eqhat \cdot \vn,\mu }-\ip{\eqhat \cdot \vn,\mu }_{\d\Omega}
   & =  0,\quad   &&\forall\, \mu\in M_h
\end{alignat}
  where
\begin{alignat}{2}
  \label{eq:err-d2}
  \eqhat\cdot\vn
                 &:= \eq\cdot\vn  + \tau (\eu -\el)
&&\quad\mbox{ on }\doh.
\end{alignat}
\end{subequations}
\end{lemma}
\emph{Proof} From \eqref{eq: weak exact solution}, we recall that  $\vq$ and $u$ satisfy the equations
  \begin{subequations}
  \begin{alignat*}{2}
    ( \vq,\vr) -
    (  u,\dive \vr) +   \ip{ u,\vr\cdot\vn}
    &  = 0 &&\quad\forall ~\vr\in \boldsymbol{V}_h,
    \\
     (\I^\alpha u',w)-(   \vq, \nabla w) +
    \ip{  \vq\cdot\vn, w}
    & =
    (f, w) &&\quad\forall\,w \in W_h\,.
\end{alignat*}
\end{subequations}
 By  \eqref{eq:proj1 new}, \eqref{eq:proj2} and the fact that $\Pm$ is the
   $L^2-$projection into $M_h$, we get
\begin{align}\label{*a}
&    ( \pv\vq,\vr) -
    ( \pw u,\dive \vr) +   \ip{\Pm u,\vr\cdot\vn}   = (e_{\vq} ,\vr),\quad \forall\,  \vr\in \boldsymbol{V}_h\\
 &   (\I^\alpha (\pw u)',w)-(\pv\vq, \nabla w)
    + \ip{  \pv\vq\cdot\vn +\tau(\pw u-\Pm u), w}\nonumber
      \\
    &  \hspace{2.2in}=
    (f+\I^\alpha e_u', w),\quad \forall\,w\in W_h, \label{*b}
  \end{align}
 given that, for each element $K\in\oh$, $\tau$ is constant on each
face $F$ of $K$.

 Subtracting the equations \eqref{eq:method-a} and
\eqref{eq:method-b} from \eqref{*a} and \eqref{*b}, respectively, we obtain
equations \eqref{eq:err-a} and \eqref{eq:err-b}, respectively.  The equation~\eqref{eq:err-c} follows directly from the
equation \eqref{eq:method-d} and \eqref{time-fractional diffusion
models-b}

By the definition of $\eqhat$ and since $P_M$ is the $L^2$-projection into $M_h$, we have
 \begin{align*}
   \ip{\eqhat\cdot\vn,\mu}&-\ip{\eqhat\cdot\vn,\mu}_{\d\Omega}
 = \ip{ (p_M\vq-\qhat_h)\cdot\vn,\mu}-\! \ip{ (p_M\vq-\qhat_h)\cdot\vn,\mu}_{
\d\Omega}\\
&=   \ip{ (\vq-\qhat_h)\cdot\vn,\mu}-\ip{
(\vq-\qhat_h)\cdot\vn,\mu}_{ \d\Omega}\\
&=   [\ip{ \vq\cdot\vn,\mu}-\ip{
\vq\cdot\vn,\mu}_{ \d\Omega}]-[\ip{\qhat_h\cdot\vn,\mu}-\ip{
\qhat_h\cdot\vn,\mu}_{ \d\Omega}]=0
\end{align*}
 where in the last equality we used  that  $\vq$ is in ${\bf  H}({\rm div},\Omega)$ and equation \eqref{eq:method-c}. Thus,  the identity
\eqref{eq:err-d} holds. For the proof of \eqref{eq:err-d2},
\begin{align*}
\eqhat\cdot\vn&=P_M
(\vq\cdot\vn)-(\vq_h\cdot\vn+\tau\,(u_h-\uhat_h))
&&\quad {\rm by}~\eqref{HDGtrace},
\\&=(\pv\vq\cdot\vn+\tau\,(\pw u-P_Mu))-(\vq_h\cdot\vn+\tau\,(u_h-\uhat_h))
&&\quad {\rm by}~\eqref{eq:proj3},
\\&=\eq\cdot\vn  + \tau (\eu -\el)\,. \quad  \Box\end{align*}
\begin{lemma}\label{lemma inter bound} Let $S_h:=\|\sqrt{\tau}(\eu -\el)\|_{\doh}$ and $c_\alpha:=\cos(\frac{\alpha\pi}{2})$. For $T>0$,
\begin{multline*} 
\int_0^T\! (\I^\alpha \eu',\eu')\,dt+\|\eq(T)\|^2+S_h^2(T) \leq
\|\eq(0)\|^2+S_h^2(0)\\+\frac{1}{c_\alpha^2}\int_0^T \!(\I^\alpha e_u',e_u')\,dt +2\int_0^T
\!(e_\vq',\eq)\,dt\,.\end{multline*}
\end{lemma}
\begin{proof}
Since $ ( \eu,\dive
 \vr)=-(\nabla\eu,\vr)+\ip{\eu,\vr.\boldsymbol{n}},$
 \eqref{eq:err-a} can be rewritten as:
\[( \eq,\vr) +(\nabla\eu,\vr) + \ip{
\el-\eu,\vr\cdot\vn}=(e_{\vq} ,\vr)\,.\]
 A time differentiation of both sides yields,
 \[( \eq',\vr) +(\nabla\eu',\vr) + \ip{ \el'-\eu',\vr\cdot\vn}=(e_{\vq}'
,\vr)\,.\]
 Setting $\vr = \eq$ and choose  $w=\eu'$ in equation \eqref{eq:err-b},  we observe,
\begin{align*}
( \eq',\eq) +(\nabla\eu',\eq) + \ip{ \el'-\eu',\eq\cdot\vn}&=(e_{\vq}'
,\eq)\,,\\
(\I^\alpha\eu',\eu')-( \eq,\nabla
\eu')+ \ip{ \eqhat \cdot\vn, \eu'}
  &= (\I^{\alpha} e_u',\eu')\,.\end{align*}
  Combine the above two equations and using $( \eq',\eq) =\frac{1}{2}\frac{d}{dt}\|\eq\|^2$, we obtain
\begin{equation} \label{eq:A1}
 (\I^\alpha \eu',\eu')+\frac{1}{2}\frac{d}{dt}\|\eq\|^2+{\psi}_{h}=(\I^{\alpha} e_u',\eu')+(e_\vq',\eq)\,.
\end{equation}
where
\[ {\psi}_{h}=  \ip{\el'-\eu',\eq\cdot\vn}+ \ip{ \eqhat
\cdot\vn, \eu'}\,.
\]
A time differentiation of \eqref{eq:err-c} follows by choosing $\mu=\eqhat \cdot\vn $ and then using \eqref{eq:err-d} yield
 $\ip{
\eqhat \cdot\vn, \el'}_{\d\Omega} =\ip{ \eqhat \cdot\vn,
\el'}=0$. Thus, by \eqref{eq:err-d2},
\begin{multline}\label{eq:psi}\psi_h=\ip{ \el'-\eu',(\eq-\eqhat)\cdot\vn}
 =\ip{\tau(\eu' -\el'),(\eu
-\el)}=\frac{1}{2}\frac{d}{dt}S_h^2(t)\,. \end{multline}
Now, integrating  \eqref{eq:A1}  over the time
interval $[0,T]$ and using \eqref{eq:psi}, we get
   \[\int_0^T \!
(\I^\alpha\eu',\eu')\,dt+\frac{1}{2}\int_0^T \!\frac{d}{dt}[\|\eq\|^2+S_h^2]\,dt=\int_0^T \!(\I^{\alpha}e_{u}',\eu')\,dt+\int_0^T
\!(e_\vq',\eq)\,dt\,.\]
 Therefore,
 \begin{multline} \label{mu:A2}  2\int_0^T\! (\I^{\alpha}\eu',\eu')\,dt+\|\eq(T)\|^2+S_h^2(T) \\=
\|\eq(0)\|^2+S_h^2(0)
+ 2\int_0^T \!(\I^{\alpha}e_{u}',\eu')\,dt+2\int_0^T
\!(e_\vq',\eq)\,dt\,.\end{multline}
An application of  the continuity property of the fractional derivative operator $\I^{\alpha}$ (see \cite[Lemma 3.6]{MustaphaSchoetzau2013}),  yields
\[2\Big|\int_0^T
\!(\I^{\alpha}e_{u}',\eu')\,dt\Big|\leq\frac{1}{c_\alpha^2}\int_0^T
\!(\I^{\alpha} e_{u}',e_u')\,dt + \int_0^T
\!(\I^{\alpha}\eu',\eu')\,dt\,.\]
Finally, inserting this in \eqref{mu:A2} and simplifying will complete the proof.
\end{proof}

  To be ready to show the main error bounds of the HDG method, we recall  from \cite[Lemma 3.6]{MustaphaSchoetzau2013} the following coercivity  property of the  operator $\I^\alpha$. For any $v \in \mathcal{C}(0,T;L_2(\Omega))$,  we have
\begin{align}&{\rm Coercivity~ property}:~~\int_0^T\!
(\I^{\alpha}v(t),v(t)) \,dt \geq c_\alpha \int_0^T\|\I^{\frac{\alpha}{2}}v(t)\|^2 \,dt\,.\label{eq: coercivity}
\end{align}
\begin{theorem}\label{thm: optimal errors}
 Assume that $u\in\mathcal{C}^1(0,T; H^{k+1}(\Omega))$ and $\vq\in\mathcal{C}^1(0,T;\boldsymbol{H}^{k+1}(\Omega))$. Assume also
that $\tau^*_K$ and $1/\tau^{\max}_K$ are bounded by $\mathsf{C}$. Then we have that
\[
\|(u-u_h)(T)\|+\|(\vq-\vq_h)(T)\|\le\, C_1(1+T)\,h^{k+1}.
\]
The constant $C_1$ only depends on  $\mathsf {C}$,  $\alpha$,
$\|u\|_{\mathcal{C}^1(H^{k+1})}$, and on $\|\vq\|_{\mathcal{C}^1(H^{k+1})}$.
\end{theorem}
\begin{proof} From the decompositions: $u-u_h=\eu-e_u$ and $\vq-\vq_h= \eq-e_\vq$, and the error projection in Theorem \ref{thm:proj}, we have
\[\|(u-u_h)(T)\|+\|(\vq-\vq_h)(T)\|\le C_1\,h^{k+1}+\|\eu(T)\|+\|\eq(T)\|\,.\]
The task now is to estimate $\|\eu(T)\|$ and $\|\eq(T)\|$. From Lemma \ref{lemma inter bound}, for $t\ge 0$, we have  $ E^2(t)\le A(t)+2\,\int_{0}^t\,B(s)\,E(s)\,ds$ where
\begin{align*}
A(t)&:=\|\eq(0)\|^2+S_h^2(0)+\frac{1}{c_\alpha^2}\int_0^t \!(\I^{\alpha}
e_{u}',e_u')\,ds, \quad B(t):=\|e_\vq'(t)\|,\\
E(t)&:=\Big(\|\eq(t)\|^2+\int_0^t\!(\I^{
 \alpha}\eu',\eu')\,ds\Big)^\frac{1}{2}\,.
\end{align*}
(Note that $A$ and $B$ are nonnegative functions.) Thus, an application of the   integral inequality (see \cite[Lemma 4]{CockburnMustapha2014}) yields  \[ E(T)\leq
\max_{t\in(0,T)} A^{\frac{1}{2}}(t)+\int_0^T\,B(s)\,ds~~{\rm for~ any}~~ T>0.\] Hence,
\begin{multline}\label{eq: before coer}
\|\eq(T)\|^2+\int_0^T(\I^\alpha \eu',\eu')\,ds \leq  C\Big(\|\eq(0)\|^2+S_h^2(0)\\
+ \int_0^T\big(\frac{1}{c_\alpha^2}\|\I^\alpha
e_{u}'\|\,\|e_u'\|
+T\|e_\vq'\|^2\big)ds\Big)\,.
 \end{multline}
However, since $\eu(t)=\int_0^t\eu'(s)\,ds=\I^{1-\frac{ \alpha}{2}} (\I^{\frac{\alpha}{2}} \eu')(t)$ (because $\eu(0)=0$), by the Cauchy-Schwarz inequality and the coercivity property of the operator $\I^\alpha$,
\begin{align*}
\|\eu(t)\|^2 &\leq \Big(\int_0^t\, \omega_{1-\frac{\alpha}{2}}(t-s)\, \|\I^{\frac{\alpha}{2}} \eu'(s)\|\,ds\Big)^2 \\
&\leq  \int_0^t\,
\omega^2_{1-\frac{\alpha}{2}}(s)\,ds \int_0^t\,\|\I^{\frac{\alpha}{2}} \eu'(s)\|^{2}\,ds
 \\
&=  \frac{t^{1-\alpha}}{(1-\alpha)\Gamma^2(1-\frac{\alpha}{2})}\int_0^t\,\|\I^{\frac{\alpha}{2}} \eu'(s)\|^{2}\,ds \\
&\leq  \frac{t^{1-\alpha}}{(1-\alpha)\Gamma^2(1-\frac{\alpha}{2})\,c_\alpha}\int_0^t\!(\I^{
 \alpha}\eu',\eu')\,ds\,.\end{align*}
 Therefore, combining \eqref{eq: before coer} with the above bound, and apply  Theorem \ref{thm:proj} for the time derivative  error projections $e_u'$ and $e_\vq'$, we obtain
 \[\|\eq(t)\|^2+ \|\eu(t)\|^2  \leq C_1^2 (1+T)^2 \Big( \|\eq(0)\|^2+S_h^2(0)+ h^{2k+2}\Big)\,.
\]
To complete the proof, we need to bound $\|\eq(0)\|^2+S_h^2(0)$.

Since $ ( \eu,\dive
 \vr)=-(\nabla\eu,\vr)+\ip{\eu,\vr.\boldsymbol{n}},$
 setting $\vr = \eq$ in \eqref{eq:err-a} and   $w=\eu'$  in \eqref{eq:err-b} yield
\begin{align*}\|\eq\|^2 +(\nabla\eu,\eq)+ \ip{ \el-\eu,\eq\cdot\vn}&=(e_{\vq},\eq),\\
(\I^\alpha\eu',\eu)-( \eq,\nabla
\eu)+ \ip{ \eqhat \cdot\vn, \eu}
  &= (\I^{\alpha} e_u',\eu)\,.\end{align*}
  Adding the above equations, and using $\ip{ \eqhat \cdot\vn,
\el}=0$ (this follows by choosing $\mu=\eqhat \cdot\vn $ in \eqref{eq:err-c} and $\mu= \el$ in \eqref{eq:err-d}) and \eqref{eq:err-d2}, yield
 \[
(\I^\alpha\eu',\eu)+\|\eq\|^2+S_h^2=(\I^{\alpha}e_{u}',\eu)+(e_\vq,\eq)\,.\]
Now, integrating
 over the time
interval $[0,t]$, observing that
\[\int_0^t \!
(\I^\alpha\eu',\eu)\,ds=\int_0^t \!
({\rm D}^{1-\alpha}\eu,\eu)\,ds \ge0,\]
(the last inequality follows from the nonnegativity  property of  the Riemann--Liouville fractional derivative operator ${\rm D}^{1-\alpha}$, see \cite[Section 2]{McLeanMustapha2009})  and using the inequality $(e_\vq,\eq) \le \frac{1}{2}\|e_\vq\|^2+\frac{1}{2}\|\eq\|^2$, we get
   \[\int_0^t [\frac{1}{2}\|\eq\|^2+S_h^2]\,ds\le\int_0^t (\|I^{\alpha}e_{u}'\|\,\|\eu\|+\frac{1}{2}\|e_\vq\|^2)\,ds\,.\]
 Therefore, by the mean value theorem for integrals, there exist $t^*, \tilde t \in (0,t) $ such that
 \[t \Big(\frac{1}{2}\|\eq(t^*)\|^2+S_h^2(t^*)\Big)\le \,t\Big(\|\eu(\tilde t)\|\max_{s\in (0,t)}\|\I^{\alpha}e_{u}'(s)\|+\frac{1}{2}\max_{t\in (0,t)}\|e_\vq(s)\|^2\Big)\,.\]
 Finally, simplify by $t$, then take $\lim_{t\downarrow 0} $ to both sides and using $\eu(0)=0$, we observe that $ \|\eq(0)\|^2+S_h^2(0)\le \|e_\vq(0)\|^2\le C_1 h^{2k+2}$ by  the error estimate of $e_\vq$ given in Theorem \ref{thm:proj}. The proof is completed now.
 \end{proof}

\section{Superconvergence and post-processing}\label{sec:super}
In this section,  we seek a better approximation to $u$ by means of an element-by-element postprocessing. We begin by describing such approximation, then we show how
to get our superconvergence result  by a duality argument.

Following \cite{ChabaudCockburn12,GastaldiNochetto89}, 
for each $t\in[0,T]$, we define  the postprocessed HDG
solution  $u_{h}^\star(t) \in \pol_{k+1}(K)$ to $u(t)$ for  each
 simplex $K\in\oh$, as follows:
\begin{subequations}
\label{eq:ustar}
  \begin{alignat}{2}
\label{eq:ustar-a}
 (u^\star_h(t) , 1 )_K
     =&\;(u_h(t), 1 )_K
  \\
\label{eq:ustar-b}
    (\nabla u^\star_h(t) , \nabla w )_K
     =&-(\vq_h(t) , \nabla w )_K
    &&\qquad \forall~ w \in \pol_{k+1}(K).
  \end{alignat}
\end{subequations}
Let $P_0$ be the $L^2(\Omega)$-projection into the space of functions which are constant on each element $K\in\oh$. One may show that
\begin{equation}
\label{ustar}  \|(u-u^\star_h)(T)\|_K  \le C\,h_{{K}}^{k+2}\,|u(T)|_{H^{k+2}(K)}+\|P_0 \eu(T)\|_K+C\,h\,\|\eq(T)\|_K.
\end{equation}
The main task now is to show that the term $\|{P_0\eu}\|$  is of order $O(h^{k+2})$. Then the postprocessed approximation $u_h^\star$ would converge
faster than the original approximation $u_h$. Noting that
$\|P_0\eu(T)\|= \sup_{{\Theta}\in   C^\infty_0(\Omega)}\frac{(P_0\eu(T),\Theta)}{\|\Theta\|}$. To estimate the expression $(P_0\eu(T),\Theta)$, we use the  traditional duality approach by using the solution of the dual problem
      \begin{align}\label{eq: dual 1}
     \vPhi + \nabla \vPsi & = 0 ~~\text{ and }~~
      (\I^{\alpha *}\vPsi)'-\dive\vPhi  = 0 ~~\text{ on } \Omega \times (0,T),
    \end{align}
 with $\vPsi=0$ on  $\d \Omega \times (0,T)$ and $\vPsi(T) =\Theta$ on  $\Omega$, where $\I^{\alpha *}$ is the adjoint operator of $\I^{\alpha}$ defined by \cite{MustaphaMcLean2013}:
\[\I^{\alpha *}\psi(t)=\int^T_t\,\omega_{\alpha}(s-t)\,\psi(s)\,ds\,.\]
 Integrating $(\I^{\alpha *}\vPsi)'-\dive\vPhi  = 0$ over the time interval $(t,T)$, we obtain
 \begin{equation}
 \label{eq: dual 8}\I^{\alpha *}\vPsi(t)+\int_t^T\dive\vPhi(s)\,ds  = 0\,.
\end{equation}
We define now  the adjoint ${\rm D}^{\alpha*}$  of the Riemann--Liouville fractional derivative operator ${\rm D}^\alpha$ (see \eqref{eq: fractional derivative} for the definition of ${\rm D}^\alpha$)  as follows  \cite{MustaphaMcLean2013}: for $t\in (0,T)$,
 \begin{align*}
 {\rm D}^{\alpha *} v(t)&=-\frac{\partial}{\partial t}\int_t^T
    \omega_{1-\alpha}(s-t)\,v(s)\,ds
\quad\text{for~any $v\in  \mathcal{C}^1(0,T)$}\,.\end{align*}
Since  $\int_t^q \omega_{\alpha}(s-t)\,\omega_{1-\alpha}(q-s)\,ds=1$,
it is easy to see that $\I^{\alpha *}$ is the {\it right-inverse} of ${\rm D}^{\alpha *}$, that is, ${\rm D}^{\alpha *}(\I^{\alpha *} \vPsi)(t)=\vPsi(t)$.
Hence, using this after
applying the operator ${\rm D}^{\alpha*}$  to both sides of \eqref{eq: dual 8}, yields
\begin{equation}
 \label{eq: dual 8 new}\vPsi(t)+{\rm D}^{\alpha *}\big(\int_t^T\dive\vPhi(q)\,dq\big)  = 0\,.
\end{equation}
However, since
\begin{align*}
 {\rm D}^{\alpha *}\big(\int_t^T\dive\vPhi(s)\,ds\big)
 &=-\frac{\partial}{\partial t}\int_t^T \omega_{1-\alpha}(s-t)\int_s^T\dive\vPhi(q)\,dq\,ds\\
 &=-\frac{\partial}{\partial t}\int_t^T \dive\vPhi(q) \int_t^q\omega_{1-\alpha}(s-t)\,ds\,dq\\
 &=-\frac{\partial}{\partial t}\int_t^T \omega_{2-\alpha}(q-t)\dive\vPhi(q) \,dq\\
 &=\int_t^T
    \omega_{1-\alpha}(s-t) \dive\vPhi(s)\,ds,
 \end{align*}
 differentiating both sides of \eqref{eq: dual 8 new} with respect to $t$, yield
 $\vPsi'-\dive {\rm D}^{\alpha *}\vPhi  = 0$.
Therefore, an alternative formulation of the  dual problem \eqref{eq: dual 1} is given by:
 \begin{subequations}
    \label{eq:dual}
    \begin{align}
      \label{eq:dual-1}
     \vPhi + \nabla \vPsi & = 0 &&\text{ on } \Omega \times (0,T), \\
      \label{eq:dual-2}
      \vPsi'-\dive {\rm D}^{\alpha *} \vPhi & = 0 &&\text{ on } \Omega \times (0,T), \\
      \label{eq:dual-3}
      \vPsi &=0  &&\text{ on } \d \Omega \times (0,T), \\
      \label{eq:dual-4}
      \vPsi(T) &=\Theta &&\text{ on } \Omega.
    \end{align}
  \end{subequations}
In the next lemma,  an expression for the quantity $(P_{0}\eu(T),\vTheta)$ in terms of the errors $\eu'$, $\eq$,  the projection errors $e_\vq$ and $e_u'$, and the solution of the
dual problem will be given. In it, $\mathrm{I}_h$ is any interpolation operator from $L^2(\Omega)$ into $W_h\cap H^1_0(\Omega)$, $\Pw$ is the $L^2$-projection
into $W_h$ and $\boldsymbol \Pi^{\mbox{{\rm\tiny{BDM}}}}$ is the well-known projection associated to the lowest-order Brezzi-Douglas-Marini
(BDM) space.

\begin{lemma}
\label{lemma:duality3}  Assume that $k\ge1$. Then, for any $T>0$,
\begin{multline*}
(P_{0}\eu(T),\vTheta)= \int_0^{T} [(\eq, {\rm D}^{\alpha *}( \nabla\mathrm{I}_h\vPsi)-\boldsymbol{\Pi}^{\mbox{{\rm\tiny{BDM}}}}\nabla\vPsi) \\+(e_{\vq}
, {\rm D}^{\alpha *}(\boldsymbol{\Pi}^{\mbox{{\rm\tiny{BDM}}}}\nabla\vPsi -\nabla \Pw\vPsi))
+(\eu'-e_u',P_{0}\vPsi-\mathrm{I}_h\vPsi) ]\,dt.
\end{multline*}
\end{lemma}
\begin{proof}
Since $\vPsi(T)=\vTheta$ by \eqref{eq:dual-4} and $\eu(0)=0$ (since $u_h(0) = \pw u_0$), we have
\begin{align*}
(P_0\eu(T),\vTheta)
  =& \int_0^{T} [((P_0\eu)',\vPsi)+(P_0\eu,\vPsi')]\,dt\\
  =& \int_0^{T} [(\eu',P_0\vPsi)+(\eu, P_0\nabla\cdot  {\rm D}^{\alpha *} \vPhi)]\,dt
\end{align*}
by the definition of the $L^2$-projection $P_0$ and by  \eqref{eq:dual-2}.

By the commutativity property
$P_0\nabla\cdot=\nabla\cdot\boldsymbol{\Pi}^{\mbox{{\rm\tiny{BDM}}}}$ and  the first error equation {\eqref{eq:err-a}} with
$\boldsymbol{r}:= {\rm D}^{\alpha *}\boldsymbol{\Pi}^{\mbox{{\rm\tiny{BDM}}}} \vPhi$ (since $k\ge 1$), we get for each $t\in (0,T]$,
\begin{subequations}
\begin{alignat*}{2}
 (\eu, P_0&\nabla\cdot  {\rm D}^{\alpha *} \vPhi)= (\eu, \nabla\cdot {\rm D}^{\alpha *}\boldsymbol{\Pi}^{\mbox{{\rm\tiny{BDM}}}} \vPhi),\\
=&(\eq, {\rm D}^{\alpha *}\boldsymbol{\Pi}^{\mbox{{\rm\tiny{BDM}}}} \vPhi) +\ip{\el, {\rm D}^{\alpha *}\boldsymbol{\Pi}^{\mbox{{\rm\tiny{BDM}}}}
\vPhi\cdot\boldsymbol{n}}
-(e_\vq, {\rm D}^{\alpha *}\boldsymbol{\Pi}^{\mbox{{\rm\tiny{BDM}}}} \vPhi)\\
=& (\eq, {\rm D}^{\alpha *}\boldsymbol{\Pi}^{\mbox{{\rm\tiny{BDM}}}} \vPhi)-(e_\vq, {\rm D}^{\alpha *}\boldsymbol{\Pi}^{\mbox{{\rm\tiny{BDM}}}} \vPhi)\\
=&(\eq, {\rm D}^{\alpha *}(-\boldsymbol{\Pi}^{\mbox{{\rm\tiny{BDM}}}}\nabla\vPsi +\nabla\mathrm{I}_h\vPsi))-(\eq, {\rm D}^{\alpha *}(\nabla\mathrm{I}_h\vPsi))-(e_\vq, {\rm D}^{\alpha *}\boldsymbol{\Pi}^{\mbox{{\rm\tiny{BDM}}}} \vPhi)\,.
\end{alignat*}
\end{subequations}
Noting that, in the second last equality we used  $$\ip{\el, {\rm D}^{\alpha *}\boldsymbol{\Pi}^{\mbox{{\rm\tiny{BDM}}}} \vPhi\cdot\boldsymbol{n}}
=\ip{\el, {\rm D}^{\alpha *}\boldsymbol{\Pi}^{\mbox{{\rm\tiny{BDM}}}} \vPhi\cdot\boldsymbol{n}}_{\partial\Omega}=0$$ which follows from
 \eqref{eq:err-d} (because  ${\rm D}^{\alpha *}\boldsymbol{\Pi}^{\mbox{{\rm\tiny{BDM}}}} \vPhi\in \Hdiv\Omega$) and the fact that $\el=0$ on $\partial \Omega$ by \eqref{eq:err-c}\,.

But, by the error equation \eqref{eq:err-b} with
$w:= {\rm D}^{\alpha *}(\mathrm{I}_h\vPsi$),
\[
(\eq, {\rm D}^{\alpha *}(\nabla\mathrm{I}_h\vPsi)) =(\I^{\alpha}(\eu'-e_{u}'), {\rm D}^{\alpha *}(\mathrm{I}_h\vPsi)) +\ip{ \eqhat
  \cdot\boldsymbol{n}, {\rm D}^{\alpha *}(\mathrm{I}_h\vPsi)}\,.\]

Now, putting together all the above intermediate steps,
\begin{multline}\label{eq: interm 1}
(P_0\eu(T),\vTheta)=
 \int_0^{T} [(\eu',P_0\vPsi)+(\eq, {\rm D}^{\alpha *}(
  \nabla\mathrm{I}_h\vPsi)-\boldsymbol{\Pi}^{\mbox{{\rm\tiny{BDM}}}}\nabla\vPsi)\\
-  ({\rm D}^\alpha\I^{\alpha}(\eu'-e_{u}'), \mathrm{I}_h\vPsi) -\ip{  \eqhat
  \cdot\boldsymbol{n}, {\rm D}^{\alpha *}(\mathrm{I}_h\vPsi)}-(e_\vq, {\rm D}^{\alpha *}\boldsymbol{\Pi}^{\mbox{{\rm\tiny{BDM}}}} \vPhi)]\,dt.
\end{multline}
Changing the order of integrals and using $\int_s^t\,\omega_{1-\alpha}(t-q) \omega_{\alpha}(q-s)\,dq=1$, we obtain the following identity:
\begin{align*}
{\rm D}^{\alpha} (\I^{ \alpha}(\eu'-e_{u}'))(t)&=\frac{\partial}{\partial t}\int_0^t\,\omega_{1-\alpha}(t-q)\int_0^q\omega_{\alpha}(q-s)(\eu-e_{u})'(s)\,ds\,dq \\
 &=\frac{\partial}{\partial t}\int_0^t\,(\eu-e_{u})'(s)\,\int_s^t\,\omega_{1-\alpha}(t-q) \omega_{\alpha}(q-s)\,dq\,ds \\ &=\frac{\partial}{\partial t}\int_0^t\,(\eu-e_{u})'(s)\,ds = (\eu-e_{u})'(t)\,. \end{align*}
In addition,
 $\ip{ \eqhat\cdot\boldsymbol{n}, {\rm D}^{\alpha *}(\mathrm{I}_h\vPsi)}= \ip{\eqhat\cdot\boldsymbol{n}, {\rm D}^{\alpha *}(\mathrm{I}_h\vPsi)}_{\partial\Omega}=0$  by \eqref{eq:err-d} and the identity $\mathrm{I}_h\vPsi=0$ on $\partial\Omega$ by the boundary
condition of the dual problem \eqref{eq:dual-3}.

Hence, by using this in  \eqref{eq: interm 1}, we observe
\begin{multline*}
(P_0\eu(T),\vTheta)=
 \int_0^{T} [(\eq, {\rm D}^{\alpha *}(
  \nabla\mathrm{I}_h\vPsi-\boldsymbol{\Pi}^{\mbox{{\rm\tiny{BDM}}}}\nabla\vPsi)) \\-(e_{\vq}
, {\rm D}^{\alpha *}\boldsymbol{\Pi}^{\mbox{{\rm\tiny{BDM}}}} \vPhi) +(\eu',P_0\vPsi-\mathrm{I}_h\vPsi) +(e_u',\mathrm{I}_h\vPsi)]\,dt.
\end{multline*}
 Therefore, the  desired result now follows after noting that
\[
-\int_0^T(e_{\vq} , {\rm D}^{\alpha *}\boldsymbol{\Pi}^{\mbox{{\rm\tiny{BDM}}}} \vPhi)\,dt =
\int_0^T(e_\vq, {\rm D}^{\alpha *}(\boldsymbol{\Pi}^{\mbox{{\rm\tiny{BDM}}}}\nabla\vPsi-\nabla \Pw\vPsi))\,dt,
\]
(by \eqref{eq:dual-1}, the fact that  $\Pw$ is the $L^2$-projection
into $W_h$, and  the
orthogonality property of the projection $\pv$, \eqref{eq:proj1 new})
and that $(e_u',\mathrm{I}_h\vPsi) =(e_u',\mathrm{I}_h\vPsi-P_0\vPsi)$ (by the fact that  $P_0\vPsi$ is constant on each element $K\in\oh$,  and the orthogonality property of the
projection $\pw$, \eqref{eq:proj2}). The proof is completed now.
\end{proof}
In the next theorem we state  the superconvergence estimate of the postprocessed HDG approximation. For the proof, we follow the derivation in \cite[Section 5]{CockburnMustapha2014} step-by-step and use Lemma \ref{lemma:duality3} instead of \cite[Lemma 7]{CockburnMustapha2014}.
\begin{theorem}
\label{thm:super} Assume that $u\in\mathcal{C}^1(0,T;H^{k+2}(\Omega))$ and $\vq\in\mathcal{C}^1(0,T;\boldsymbol{H}^{k+1}(\Omega))$. Assume also
that $\tau^*_K$ and $1/\tau^{\max}_K$ are bounded by $\mathsf{C}$. Then,   we have
\begin{subequations}
\begin{alignat*}{1}
\|(u-u^*_h)(T)\|&\le\, C_2\,  \max\{1, \sqrt{\log(T h^{-2/(\alpha+1)})}\}\,h^{k+2}\quad {\rm for}~~k\ge 1,
\end{alignat*}
\end{subequations}
 where the constant $C_2$,
  only depends on  $\mathsf {C}$, $\alpha$, $T$,  $\|u\|_{\mathcal{C}^1(H^{k+2})}$, and on $\|\vq\|_{\mathcal{C}^1(H^{k+1})}$.
 \end{theorem}

\section{Numerical experiments}\label{sec: numerics}
In this section, we present numerical experiments devised to validate our theoretical predictions from HDG spatial discretizations. To do so, we use the fully discrete CN HDG scheme \eqref{Fully discrete scheme}.
 We take the (uniform) time steps $\delta$ to be sufficiently small  so that the HDG and postprocessed HDG spatial discretizations errors  are dominant.
 This is achieved by fixing the ratio $\frac{\delta^2}{h^{k+2}}$ to a given number less than the unit  because the time stepping CN scheme is second-order accurate provided that the exact solution is sufficiently regular.

We choose the spatial domain $\Omega$ to be the unit interval $(0,1)$  and $T = 1$ in \eqref{eq: time-fractional diffusion models}. We impose homogenous Dirichlet
boundary conditions and choose the source term $f$ and the initial data $u_0$ so that the exact solution is $u(x,t)=t^{3-\alpha}sin(\pi x)\,.$ For different values of $\alpha$, we obtain the history of convergence of the errors $\|(u-u_h)(T)\|$, $\|(\vq-\vq_h)(T)\|$ and $\|(u-u^\star_h)(T)\|$ for different values of the polynomial degree, $k = 0, 1, 2.$ To compute the spatial $L_2$-norm, we apply a composite Gauss quadrature rule with $4$-points on each interval of the finest spatial mesh. The numerical results (errors and convergence rates) of the experiments are
presented in Tables \ref{tab: errors for alpha=0.5} and \ref{tab: errors for
  alpha=0.7}. In full agreement with our theoretical results, we obtain optimal convergence rates for the HDG scheme and $O(h^{k+2})$ superconvergence rates for the postprocessed HDG scheme.
\begin{table}[htb]
\renewcommand{\arraystretch}{1}
\begin{center}
\begin{tabular}{|r|rr|rr|rr|}
 \hline
$N$ & \multicolumn{6}{c|}{$k=0$}\\
\hline
  4& 5.269e-01&      & 7.899e-01&      & 5.048e-01&      \\
  8& 3.027e-01& 0.799& 4.028e-01& 0.972& 2.922e-01& 0.788\\
 16& 1.616e-01& 0.905& 2.025e-01& 0.992& 1.566e-01& 0.899\\
 32& 8.342e-02& 0.954& 1.014e-01& 0.997& 8.098e-02& 0.951\\
 64& 4.237e-02& 0.977& 5.072e-02& 0.999& 4.117e-02& 0.976\\
128& 2.135e-02& 0.989& 2.537e-02& 0.999& 2.076e-02& 0.988\\
\hline
\hline & \multicolumn{6}{c|}{$k=1$}\\
\hline
  4& 6.031e-02&      & 5.936e-02&      & 7.401e-03&      \\
  8& 1.502e-02& 2.005& 1.321e-02& 2.165& 8.835e-04& 3.066\\
 16& 4.144e-03& 1.858& 3.487e-03& 1.924& 1.142e-04& 2.951\\
 32& 1.048e-03& 1.983& 8.649e-04& 2.011& 1.420e-05& 3.008\\
   \hline
\hline & \multicolumn{6}{c|}{$k=2$}\\
\hline
  4& 3.960e-03&      & 4.596e-03&      & 8.902e-04&     \\
  8& 5.059e-04& 2.969& 4.868e-04& 3.239& 5.497e-05& 4.017\\
 16& 6.352e-05& 2.993& 5.652e-05& 3.107& 3.416e-06& 4.008\\
 32& 7.957e-06& 2.997& 7.117e-06& 2.989& 2.153e-07& 3.988\\
   \hline
 \end{tabular}
   \vspace{0.05in} \caption {The errors $\|(u_h-u)(T)\|$, $\|(\vq_h-\vq)(T)\|$ and $\|({u_h^{\star}}-u)(T)\|$, and the corresponding rates of convergence  for $\alpha=0.5$ with HDG solutions of degree $k=0,1,2$.  We observe optimal convergence of
   order $h^{k+1}$  for the errors in $u_h$ and $\vq_h$, and superconvergence rates of order $h^{k+2}$ (when $k\ge 1$) for the error from the postprocessed HDG solution $u_h^\star.$}
   \label{tab: errors for alpha=0.5}
\end{center}
\end{table}

\begin{table}[htb]
\renewcommand{\arraystretch}{1}
\begin{center}
\begin{tabular}{|r|rr|rr|rr|}
 \hline
$N$ & \multicolumn{6}{c|}{$k=0$}\\
\hline
  4& 5.455e-01&      & 7.705e-01&      & 5.240e-01&      \\
  8& 3.122e-01& 0.805& 3.088e-01& 9.884& 3.020e-01& 0.795\\
 16& 1.661e-01& 0.910& 1.939e-01& 1.002& 1.612e-01& 0.905\\
 32& 8.558e-02& 0.957& 9.674e-02& 1.003& 8.320e-02& 0.954\\
 64& 4.342e-02& 0.979& 4.830e-02& 1.002& 4.225e-02& 0.978\\
128& 2.187e-02& 0.989& 2.413e-02& 1.001& 2.128e-02& 0.989\\
\hline
\hline & \multicolumn{6}{c|}{$k=1$}\\
\hline
  4& 6.081e-02&      & 6.005e-02&      & 7.898e-03&      \\
  8& 1.501e-02& 2.018& 1.321e-02& 2.185& 9.403e-04& 3.070\\
 16& 4.154e-03& 1.854& 3.485e-03& 1.922& 1.218e-04& 2.949\\
 32& 1.048e-03& 1.987& 8.434e-04& 2.047& 1.506e-05& 3.015\\
  \hline
\hline & \multicolumn{6}{c|}{$k=2$}\\
\hline
  4& 4.025e-03&      & 4.978e-03&      & 1.079e-03&      \\
  8& 5.088e-04& 2.984& 5.014e-04& 3.312& 6.698e-05& 4.010\\
 16& 6.367e-05& 2.998& 5.701e-05& 3.137& 4.167e-06& 4.007\\
 32& 7.892e-06& 3.012& 7.031e-06& 3.019& 2.641e-07& 3.980\\
      \hline
 \end{tabular}
   \vspace{0.05in} \caption {The errors $\|(u_h-u)(T)\|$, $\|(\vq_h-\vq)(T)\|$ and $\|({u_h^{\star}}-u)(T)\|$, and the corresponding rates of convergence  for $\alpha=0.7$ with HDG solutions of degree $k=0,1,2$. } \label{tab: errors for alpha=0.7}
\end{center}
\end{table}


\begin{thebibliography}{10}


\bibitem{ChabaudCockburn12}  B.~Chabaud and B.~Cockburn,
Uniform-in-time superconvergence of {HDG} methods for the heat equation, {\em Math. Comp.}, 81, (2012) 107--129.

\bibitem{ChenLiuAnhTurner2011}
C-M. Chen, F. Liu, V. Anh and I. Turner,
Numerical methods for solving a two-dimensional variable-order anomalous sub-diffusion equation,
{\em Math.  Comp.}, 81, (2012) 345-366.

\bibitem{ChenLiuLiuChenAnhTurnerBurrage2014} J. Chen, F. Liu, Q. Liu, X. Chen, V. Anh, I. Turner, K. Burrage, Numerical simulation for the three-dimension fractional sub-diffusion equation, {\em Appl. Math. Model.}, 38, (2014) 3695-–3705

\bibitem{CockburnGopalakrishnanLazarov09}
B.~Cockburn, J.~Gopalakrishnan and R.~Lazarov, {Unified hybridization of
  discontinuous {Galerkin}, mixed and continuous {Galerkin} methods for second
  order elliptic problems}, {\em SIAM J. Numer. Anal.}, {47}, (2009)
  1319--1365.

\bibitem{CockburnGopalakrishnanSayas09}
B.~Cockburn, J.~Gopalakrishnan and F.-J. Sayas, {A projection-based error
  analysis of {HDG} methods}, {\em Math. Comp.}, {79}, (2010) 1351--1367.

\bibitem{CockburnMustapha2014}
B. Cockburn and K. Mustapha, A hybridizable discontinuous Galerkin  method  for fractional diffusion problems, {\em Numer. Math.}, (2014), DOI 10.1007/s00211-014-0661-x.


\bibitem{CockburnQiuShi}
B.~Cockburn, W.~Qiu and K.~Shi, {Conditions for superconvergence of {HDG}
  methods for second-order elliptic problems}, {\em Math. Comp.}, ~81, (2012) 1327-–1353.


\bibitem{CuestaLubichPalencia2006}
E. Cuesta, C. Lubich and C. Palencia, Convolution quadrature time discretization of fractional diffusive-wave equations, {\em Math. Comp.},~{75},
(2006) 673--696.

\bibitem{Cui2009} M. Cui, Compact finite difference method for
             the fractional diffusion equation, {\em J. Comput. Phys.}, 228,
 (2009)
7792--7804.

\bibitem{Cui2012}
M. Cui, Convergence analysis of high-order compact alternating direction implicit schemes for the two-dimensional time fractional diffusion
equation, {\em Numer. Algor.}, 62, (2013) 383�-409.


\bibitem{GaoSun2011} G.G. Gao and Z.Z. Sun, A box-type scheme for fractional sub-diffusion
equation with Neumann boundary conditions, {\em J. Comput. Phys.}, 230, (2011) 6061-�6074.

\bibitem{GastaldiNochetto89}
L.~Gastaldi and R.H. Nochetto, {Sharp maximum norm error estimates for
  general mixed finite element approximations to second order elliptic
  equations}, {\em RAIRO Mod\'el. Math. Anal. Num\'er.}, {23}, (1989) 103--128.





 \bibitem{JinLazarovZhou2014} B. Jin, R. Lazarov, and Z. Zhou, On two schemes for fractional diffusion and diffusion-wave equaiton,  http://arxiv.org/pdf/1404.3800v1.pdf, (2014).

\bibitem{kilbas} A.A. Kilbas, H.M. Srivastava and  J.J. Trujillo, Theory and Applications of Fractional Differential
Equations, Volume 204 (North-Holland Mathematics Studies), 2006.


\bibitem{KirbySherwinCockburn}
R.M. Kirby, S.J. Sherwin and B.~Cockburn, {To {HDG} or to {CG}: A
  comparative study}, {\em J. Sci. Comput.}, 51, (2012) 183--212.

\bibitem{LiXu2009} X. Li and C. Xu, A space-time spectral method for the time fractional diffusion equation,
{\em SIAM J. Numer. Anal.}, 47, (2009), 2108--2131.


\bibitem{LinXu2007} Y. Lin and C. Xu, Finite differnce/spectral approximations for the time-fractional diffusion equation,
{\em J. Comput. Phys.}, 225, (2007), 1533--1552.



\bibitem{McLean2010} W. McLean, Regularity of solutions to a time-fractional diffusion
equation, \textit{ANZIAM J.}, {52},  (2010) 123--138.

\bibitem{McLeanMustapha2009}
W. McLean and K. Mustapha, Convergence analysis of a discontinuous Galerkin method for a sub-diffusion equation,
 {\em Numer. Algor.}, {52},
(2009) 69--88.

\bibitem{McLeanMustapha2014}
W. McLean and K. Mustapha, {Error analysis of a discontinuous Galerkin method for a fractional diffusion equation with a non-smooth initial data},
 {\em J. Comput. Phys.}, (2014), DOI: 10.1016/j.jcp.2014.08.050.


\bibitem{MetzlerKlafter2000}
R. Metzler and J. Klafter, The random walk's guide to anomalous diffusion: a fractional dynamics approach, {\em Physics Reports} {339}, (2000)
1--77.

\bibitem{Mustapha2011}
K. Mustapha, An implicit finite difference time-stepping method for
            a sub-diffusion equation, with spatial discretization
            by finite elements, {\em IMA J. Numer. Anal.}, {31}, (2011) 719--739.

\bibitem{Mustapha2014}	K. Mustapha, Time-stepping discontinuous Galerkin methods
for fractional diffusion problems, {\em  Numer. Math.}, (2014), DOI 10.1007/s00211-014-0669-2.

\bibitem{MustaphaAlMutawa}
K. Mustapha and J. AlMutawa, A finite difference  method for an anomalous sub-diffusion equation, theory and applications, {\em Numer. Algor.},
 61, (2012) 525--543.

\bibitem{MustaphaAbdallahFurati2014} K. Mustapha, B. Abdallah and K. Furati,
A discontinuous Petrov-Galerkin method for time-fractional diffusion equations, {\em SIAM J. Numer. Numer. Anal.}, 52 (2014) 2512--2529.


\bibitem{MustaphaMcLean2011}
K. Mustapha and  W. McLean, Piecewise-linear, discontinuous Galerkin method for a fractional diffusion equation, {\em Numer. Algor.}, {56},
(2011) 159--184.

\bibitem{MustaphaMcLean2012}
K. Mustapha and  W. McLean,  Uniform convergence for a discontinuous Galerkin, time stepping method applied to a fractional diffusion equation, {\em IMA J. Numer.
Anal.}, {32},  (2012) 906--925.

\bibitem{MustaphaMcLean2013}
K. Mustapha and  W. McLean, Superconvergence of a discontinuous Galerkin method for fractional diffusion and wave equations, {\em  SIAM J. Numer. Anal.},  51,
(2013) 491--515.


\bibitem{MustaphaSchoetzau2013} K. Mustapha and  D. Sch\"otzau, Well-posedness of $hp-$version discontinuous Galerkin methods for fractional diffusion wave equations, {\em  IMA J. Numer. Anal.}, 34, (2014) 1426--1446.

 \bibitem{NguyenPeraireCockburnIcosahom09}
N.C. Nguyen, J.~Peraire and B.~Cockburn, {Hybridizable discontinuous
  {Galerkin} methods}, {\em Proceedings of the International Conference on Spectral
  and High Order Methods} (Trondheim, Norway), Lect. Notes Comput. Sci. Engrg.,
  Springer Verlag, June 2009.

\bibitem{Quintana-MurilloYuste2011} J. Quintana-Murillo and S.B. Yuste, An explicit difference method for solving
fractional diffusion and diffusion-wave equations in the Caputo form, {\em  J. Comput. Nonlin. Dyn.}, 6, (2011) 021014.



\bibitem{SweilamKhaderMahdy2012}  N. H. Sweilam, M. M. Khader and A. M. S. Mahdy, Crank-Nicolson finite difference method for solving
time-fractional diffusion equation, {\em J. Fract. Cal. Appl.}, 2 (2012) 1--9.

\bibitem{Wyss1986}
W. Wyss, Fractional diffusion equation, {\em J. Math. Phys.}, {27}, (1986)
 2782--2785.

\bibitem{XuZheng2013} Q. Xu and Z. Zheng, Discontinuous Galerkin method for time fractional
diffusion equation, {\em J. Informat. Comput. Sci.}, 10 (2013) 3253–-3264





\bibitem{YusteQuintana2009}
S.B. Yuste and J. Quintana-Murillo, On Three Explicit Difference Schemes for Fractional Diffusion and Diffusion-Wave Equations, {\em Phys. Scripta
T136}, (2009) 014025.

\bibitem{ZengoLiLiuTurner2013} F. Zengo, C. LI, F. Liu, I. Turner, The use of finite difference/element approaches for solving the
time-fractional subdiffusion equation, {\em  SIAM J. Sci. Comput.}, {35}, (2013) A2976--A3000.

\bibitem{ZhangSun2011} Y.-N. Zhang and
Z.-Z. Sun, Alternating direction implicit schemes for the two-dimensional fractional sub-diffusion equation,
 {\em  J. Comput. Phys.}, {230}, (2011) 8713--8728.



\bibitem{ZhaoSun2011} X. Zhao and Z-z. Sun, A box-type scheme for fractional sub-diffusion equation with Neumann boundary conditions,
{\em  J. Comput. Phys.}, 230, (2011) 6061--6074.

\bibitem{ZhaoXu2014} X. Zhao and Q. Xu,  Efficient numerical schemes for fractional sub-diffusion
equation with the spatially variable coefficient, {\em Appl. Math. Model.}, 38, (2014) 3848--3859.

\end{thebibliography}
\end{document}